\documentclass[reqno,12pt]{amsart}

\usepackage[utf8]{inputenc}

\usepackage{enumerate}
\usepackage[margin=1in,marginparwidth=0.7in]{geometry}

\usepackage{tikz-cd}
\usepackage{ifpdf}
\usepackage{amsmath}
\usepackage{amsfonts}
\usepackage{amssymb}
\usepackage{amsthm}
\usepackage{comment}
\usepackage[ocgcolorlinks,hyperfootnotes=false,colorlinks=true,citecolor=blue,linkcolor=blue,urlcolor=blue]{hyperref}
\usepackage{setspace}
\usepackage{amsrefs}
\usepackage{nicefrac}
\usepackage{graphicx}
\usepackage{color}
\usepackage{mathtools}

\newcommand{\ignore}[1]{}

\renewcommand{\Im}{\operatorname{Im}}

\newcommand{\abs}[1]{\left\lvert {#1} \right\rvert}

\newcommand{\C}{{\mathbb{C}}}
\newcommand{\R}{{\mathbb{R}}}

\newcommand{\N}{{\mathbb{N}}}

\newcommand{\D}{{\mathbb{D}}}

\newcommand{\sE}{{\mathcal{E}}}

\newcommand{\sH}{{\mathcal{H}}}

\newcommand{\sO}{{\mathcal{O}}}

\newtheorem{thm}{Theorem}[section]

\newtheorem{prop}[thm]{Proposition}

\newtheorem{cor}[thm]{Corollary}    

\newtheorem{lemma}[thm]{Lemma}

\theoremstyle{definition}
\newtheorem{defn}[thm]{Definition}

\theoremstyle{remark}
\newtheorem{remark}[thm]{Remark}

\author{Ji\v{r}\'{\i} Lebl}
\thanks{The first author was in part supported by Simons Foundation collaboration grant 710294.}
\address{Department of Mathematics, Oklahoma State University,
Stillwater, OK 74078, USA}
\email{lebl@okstate.edu}

\author{Luka Mernik}
\address{Department of Applied Mathematics, Florida Polytechnic University,
Lakeland, FL 33805, USA}
\email{lmernik@floridapoly.edu}

\date{\today}
    
\title{On Segre-degenerate Levi-flat hypervarieties}


\begin{document}

\begin{abstract}
We prove that a singular real-analytic Levi-flat hypersurface $H$ in $\C^n$ being
Segre-degenerate at a point $p$ is equivalent to the existence of a so-called
support curve, that is, a holomorphic curve that intersects $H$ at exactly
one point, which in turn is equivalent to the existence of support curves on at least
two sides of $H$ at $p$.  The existence of such two-sided support provides families
of analytic discs attached to $H$ that cover a neighborhood of $p$.  The existence
of such discs has two corollaries.  First,
any function holomorphic on a neighborhood of a Segre-degenerate $H$ extends to a fixed
neighborhood of $p$.  Second, the rational hull of $H$ is a neighborhood of $p$,
and thus no Levi-flat Segre-degenerate hypersurface in $\C^n$ can be rationally 
convex.
\end{abstract}

\maketitle



\section{Introduction} \label{section:intro}

A smooth hypersurface $H \subset \C^n$ is called Levi-flat if it is pseudoconvex from both sides, and one can prove that it is foliated by complex hypersurfaces.
If $H$ is real-analytic (and nonsingular), then $H$ is locally
biholomorphically equivalent to one given by $\Im z_1 = 0$.  In this case, any small analytic disc attached to $H$ lies entirely in $H$.
Interestingly, when $H$ is singular, this is no longer the case.
However, every real-analytic subvariety that is Levi-flat at regular points
does divide (locally) $\C^n$ into two or more pseudoconvex domains, and is
therefore ``Levi-flat'' even at the singular point (see e.g.,\ \cite{singularSetLeviFlat}).

Consider the hypersurface $H \subset \C^2$ defined by $\abs{z}^2-\abs{w}^2=0$.
This ``cone'' is a Levi-flat hypersurface with a singular point at the origin,
which is a so-called Segre-degenerate point, that is, the Segre variety degenerates there.
The Segre varieties outside the origin are precisely the complex lines through the origin.
Any one of the complex lines through the origin not contained in $H$ intersects $H$ only at the origin.
For a nonsingular $H$, such behavior would mean that $H$ is strictly pseudoconvex, however,
for this singular $H$, if we translate this complex line by a small amount in any direction it intersects
$H$ in a circle or an ellipse.  This behavior is reminiscent of what happens
for a strictly pseudoconvex nonsingular hypersurface when
translating in the pseudoconvex direction.  Through the translations, we obtain analytic discs attached to $H$. 
By taking two different lines in two different components of the complement and translating those, we fill an entire neighborhood of the origin
with continuous families of discs that shrink down to the origin.
In this case, the lines $\{ w = s \}$ and $\{ z = s \}$
for various constants $s$ are sufficient.
The rational hull of $H$ is thus $\C^2$, and even locally, the rational hull
of any compact neighborhood $K \subset H$ of the origin is a neighborhood of $0$
in~$\C^n$.
Locally, the two sides of the cone $H$ are Hartogs triangles.  What we prove
in this paper is that all Segre-degenerate Levi-flat hypersurfaces cut space
into pseudoconvex sets that behave like the Hartogs triangle in a certain sense.

We will investigate Levi-flat real-analytic subvarieties $H \subset \C^n$ of codimension 1.
We are interested in local behavior, and so we assume below that $H \subset U$ where $U$ is a small
neighborhood (most often a polydisc) of a point $p \in H$ and $H$ is closed in $U$.
We say that $H$ is \emph{curve-supported} at~$p$,
if there exists a holomorphic curve that only intersects $H$ at~$p$,
and we call such a curve a \emph{support curve}.
If $H$ is curve-supported at $p$ with two curves which lie in
two different components of the complement $U \setminus H$, then we say $H$
is \emph{curve-support from two sides} at $p$.
If $\rho$ is a defining function for $H$ at $p$, then the Segre variety at $p$
is the variety given by $\rho(z,\bar{p}) = 0$, and $H$ is
Segre-degenerate at $p$ if $\rho(z,\bar{p}) \equiv 0$.  See
sections \ref{section:prelim} and \ref{section:sc} for the rather subtle
details of these definitions in the singular setting.

Local properties of real-analytic singular Levi-flat hypersurfaces
and their study using Segre varieties was started by Burns--Gong~\cite{burnsgong:flat}. Many other researchers have made recent contributions, see, among others,
Brunella~\cite{foliationExt},
Cerveau--Lins Neto~\cite{CerveauLinsNeto11},
Fern\'{a}ndez-P\'{e}rez~\cite{FP14},
Shafikov--Sukhov~\cites{SS15,hullsLeviFlat},
Pinchuk--Shafikov--Sukhov~\cite{dicritical},
Beltr\'{a}n--Fern\'{a}ndez-P\'{e}rez--Neciosup~\cite{BFN:18},
Lebl~\cite{singularSetLeviFlat},
and the references within.
It is not difficult to see that, locally, a nonsingular
Levi-flat hypersurface is polynomially convex.
Shafikov and Sukhov \cite{hullsLeviFlat} showed that a
Levi-flat Segre-degenerate hypersurface cannot be polynomially or
rationally convex. In this paper, we extend their result in the
Segre-degenerate case. 
We also recover the result from Pinchuk--Shafikov--Sukhov~\cite{dicritical}
that for a Levi-flat hypersurface Segre-degenerate at $p$,
infinitely many leaves of the Levi foliation contain $p$,
that is, $p$ is a dicritical point.

With the example above in mind, an interesting question becomes:
Which properties of the cone are responsible for the existence of support curves in two components of the complement?
The main theorem of this paper answers this question for Levi-flat hypersurfaces. Namely,

    \begin{thm}\label{thm:main}
        Let $H \subset U \subset \C^n$ be a singular real-analytic Levi-flat hypersurface, $p \in H$,
        and the germ of $H$ is irreducible at $p$.
        The following are equivalent
        \begin{enumerate}
            \item $p$ is a Segre-degenerate point of $H$,
            \item $H$ is curve-supported at $p$,
            \item $H$ is curve-supported from two sides at $p$, and
            \item $p$ is a dicritical singularity of $H$.
        \end{enumerate}
    \end{thm}
    
As we mentioned,
boundaries of smooth strictly pseudoconvex domains are 
curve-supported, but cannot be curve-supported from two sides.
If a hypersurface is curve supported from two sides, then we obtain a large
rational hull.

\begin{prop} \label{prop:scopenhull}
Suppose $H \subset U \subset \C^n$ is a singular real-analytic hypersurface that is curve supported
from two sides at $p \in H$.  Then for any compact neighborhood $K \subset H$ of $p$, the polynomial
hull $\hat{K}$ and the rational hull $\hat{K}_\text{rat}$ is a
neighborhood\footnote{By ``neighborhood of $p$,'' we mean, as usual, ``contains an open set containing $p$.''}
of $p$ in $\C^n$.
\end{prop}

Hence, the polynomial hull and the rational hull of any Levi-flat in $\C^n$,
Segre-degenerate at $p$, is a neighborhood of $p$.
In particular, a Segre-degenerate Levi-flat hypersurface in $\C^n$,
despite being a union of complex hypersurfaces,
will never be polynomially nor rationally convex.

The construction of the hull using the support curves does not simply yield
analytic discs; it yields families of analytic discs shrinking to $p$.
Therefore, applying (a version of) Kontinuit\"{a}tssatz gives us the
following corollary for extending holomorphic functions.

\begin{cor}\label{cor:extension}
Suppose $H \subset U \subset \C^n$ is a singular real-analytic hypersurface that is
curve supported from two sides at $p \in H$.
Then there exists a fixed neighborhood $\Omega$ of $p$ in $\C^n$, such that
every function in $f \in \sO(H)$ extends to a holomorphic function on $\Omega$.
\end{cor}

By $\sO(H)$ we mean restrictions of holomorphic functions;
each such $f$ is holomorphic in some neighborhood of $H$.
The key point of the corollary is that $\Omega$ is
a fixed neighborhood of $p$.

The structure of the paper is the following.
In Section \ref{section:prelim},
we establish notation and summarize well-known facts.
Sections \ref{section:sc} and \ref{section: Cn} are used to prove \ref{thm:main} by constructing desired curves for hypersurfaces in questions. As an application, we prove Proposition \ref{prop:scopenhull} and Corollary \ref{cor:extension} in Section \ref{section:polyhulls}.

\section{Preliminaries} \label{section:prelim}

Suppose $U \subset \C^n$ and $\rho \colon U \to \R$ is a real-analytic
function.
Let $H=\{z\in U:\rho(z,\bar z)=0\}\subset U$ be a real-analytic
singular hypersurface.
When we write ``$H \subset \C^n$ is a hypersurface,''
we will have such a setup in mind.
A point $p\in H$ is regular if there is a
neighborhood $V$ of $p$ such that $H\cap V$ is a real-analytic submanifold.
The set of regular points of $H$ is denoted $H^{\text{reg}}$ and its
complement, the set of singular points, by $H^{\text{sing}}=H\setminus
H^{\text{reg}}$. 

The dimension of $H$ at $p\in H^{\text{reg}}$, denoted by $\dim_p H$, is the
real dimension of the real-analytic manifold at $p$,
and $\dim H$, the
dimension of $H$, is the maximum dimension at any regular point.
Denote by $H^*=\{p\in H^{\text{reg}}: \dim_p H=\dim H \}$ the
set of top-dimensional points, 
for us, $\dim H^* = \dim H = 2n-1$.
Even irreducible real-analytic sets do not always have pure
dimension. A classic example of such variety are
the so-called Cartan or Whitney umbrellas.
Moreover, while we can take $\rho$ to be the defining function at some $p$,
that is, we can assume that $\rho$ generates the ideal of germs at $p$ of
real-analytic functions vanishing on $H$, we cannot guarantee that
$\rho$ generates this ideal at points arbitrarily near $p$.
In particular, it is not true in general that $H^*$ is the set of points
where the gradient of $\rho$ is nonzero.

Holomorphic functions on $H$ are denoted by $\sO(H)$. A function $f\in
\sO(H)$ if $f$ is holomorphic in some neighborhood of $H$.  The envelope of
holomorphy of a domain $U$, denoted by $\sE(U)$, is the largest branched
domain (holomorphic covering space with finite fibers)
such that every holomorphic function on $\sE(U)$ extends to a holomorphic
function on $U$. Therefore, $\sE(U)$ is itself a domain of holomorphy. 

Throughout the paper $B_r(p)=\{z\in \C^n: ||z-p||<r\}$ denotes the open ball of radius $r$ centered at $p$.

\subsection{Levi-flat hypersurfaces}

Regular hypersurface $H$ is called Levi-flat if it is pseudoconvex from both
sides.  A singular hypersurface $H$ is Levi-flat if the set of regular
points of top dimension, $H^*$, is Levi-flat.
A ``model'' Levi-flat hypersurface is
\begin{equation}
M_n=\{(z_1,\dots,z_{n})\in \C^n : \Im z_1=0\}.
\end{equation}
For every $p \in H^{*}$, there is a neighborhood $U$ of $p$ and a biholomorphism
$\phi \colon U \to V$ such that $\phi(H\cap U)=M_n\cap V$. See, for
instance,~\cite{realsubbook}.

This biholomorphism induces a foliation of $H^{*}$ by complex hypersurfaces
called the Levi foliation.
The natural question arises: Does this foliation extend holomorphically to
some neighborhood of $H$? 
The first author~\cite{singularSetLeviFlat} and independently
Cerveau--Lins Neto~\cite{CerveauLinsNeto11} showed that if the singular set
is small enough, the Levi foliation extends as a singular holomorphic
foliation in a neighborhood of the hypersurface.  In general, an extension
of the foliation is not possible, but one can extend to a web, see
Brunella~\cite{foliationExt}.

\subsection{Complexification and Segre varieties}

Let $\Delta \subset \C^n$ be a polydisc centered at $0$,
suppose $\rho \colon \Delta \to \R$ is a real-analytic function vanishing
at the origin, and suppose that the power series of $\rho$ at the origin
converges in $\Delta$.  The power series $\rho(z,\bar{z})$ then converges
in $\Delta \times \Delta$ if we treat $\bar{z}$ as a separate variable, that
is, $\rho(z,w)$ is well-defined for $(z,w) \in \Delta \times \Delta$.
We remark that it is possible to use different neighborhoods than a
polydisc where the power series converges,
but as we are interested in local properties, it is sufficient
to assume the simple geometry of the polydisc and to assume it is
sufficiently small.

Our hypersurface is defined by
\begin{equation}
H = \{ z \in \Delta : \rho(z, \bar{z})=0\},
\end{equation}
and we will assume that $\rho$ generates the ideal of real-analytic
germs at $0$ vanishing on $H$.

The complexification of $H$ is 
\begin{equation}
{\mathcal{H}}=\{(z,w)\in \Delta\times\Delta : \rho(z, w)=0\}.
\end{equation}
The hypersurface $H$ lifts to the diagonal of the complexification by $i(H)$,
where $i(z)=(z,\bar{z})$.
The complexification is equipped with two projections
$\pi_z(z,w)=z$ and $\pi_w(z,w)=w$.
\begin{equation}
\begin{tikzcd}
&\mathcal H \arrow[ld,"\pi_z"] \arrow[rd,"\pi_w"] &\\
\C_z^n&& \C_w^n
\end{tikzcd}
\end{equation}
Note that we take $\pi_z$ and $\pi_w$ to have $\sH$ as their domain.
Using these projections, we define Segre varieties.

\begin{defn}
The Segre variety (with respect to $\rho$) at $q$ is given by
\begin{equation}
\Sigma_q=\pi_z\circ \pi_w^{-1}(\bar q).
\end{equation}
For a hypersurface, this is equivalent to 
\begin{equation}
\Sigma_q=\{z\in \Delta: \rho(z,\bar q)=0\} .
\end{equation}
\end{defn}

$H$ need not be a so-called coherent variety, and in particular
$\rho$ need not generate the ideal of germs vanishing on $H$ at all points.
The germ of the Segre variety at the origin is well-defined.
However, the germ of the Segre variety at some other point $q \in H$ could
be in fact smaller than the germ $\Sigma_q$
if we took a small neighborhood of $q$ and repeated the construction above.
See e.g.,~\cite{Lebl:semianal} for more on these issues. 
For the purposes of our work here, we will only take the Segre varieties
with respect to $\rho$.

Segre varieties may have multiple components.
We will call the component of $\Sigma_p$ that contains $0$ the principal
component of $\Sigma_p$.

The reality of $\rho$ implies the following
\begin{enumerate}
    \item $p\in \Sigma_p$ if and only if $p\in H$, and
    \item $p\in \Sigma_q$ if and only if $q\in \Sigma_p$.
\end{enumerate}

Segre varieties are a widely used tool for dealing with real-analytic
submanifolds in complex manifolds. Their application to CR geometry was
popularized by Webster~\cite{onmapping} and
Diederich and Fornaess~\cite{pcxwithrealbound}.
When applying these techniques to singular
varieties, one must be very careful.
Firstly, points can be degenerate,
meaning the Segre varieties are of ``wrong'' dimension.
Secondly, as we mentioned above, variety may not be coherent,
meaning the Segre varieties cannot be defined
by the same function at all points. 

\begin{defn}
A point $p\in H\subset \Delta$ is \emph{Segre-degenerate} (with respect to $\rho$)
if $\dim \Sigma_p=n$.

A point $p$ is a \emph{dicritical singularity}
if $p$ belongs to the closure of infinitely many geometrically different
leaves of the Levi foliation of $H^*$.
\end{defn}

A point may be Segre-degenerate with respect to one function and not with
respect to another.  If one defines \emph{Segre-degenerate} for germs
(and hence independently of $\Delta$ and $\rho$) using
the defining function at each point, then one may obtain fewer such points.
However, we are assuming that $\rho$ is the defining function for the
ideal of $H$ at $0$ and thus $H$ is Segre-degenerate in the germ
sense at the origin if and only if $H$ is Segre-degenerate at $0$ with respect to
$\rho$.  The set of Segre-degenerate points in the germ sense
is a semianalytic subset that is contained in a complex analytic
subvariety of dimension at most $n-2$, see~\cite{Lebl:semianal}.
On the other hand, it is easy to see that the
set of Segre-degenerate points with respect to $\rho$ is in fact
a complex analytic subvariety of at most dimension $n-2$.
In particular, in $\C^2$ Segre-degenerate points are isolated.

If $H$ is Levi-flat, some components of Segre varieties are easier to
describe. More specifically, a leaf of the Levi foliation near a regular
point is contained in a component of a Segre variety.  For this reason, a
dicritical point $p$ must be Segre-degenerate in the germ sense, and hence
also with respect to $\rho$.

\subsection{Polynomial Hulls, Rational Hulls, and Attached Discs}

\begin{defn}
Given a (pre)compact set $H$, the polynomial hull of $H$ is
\begin{equation}
\hat{H}=\left\{ z\in \C^n : |p(z)|\leq \sup_{w\in H}|p(w)|, \text{ where $p$ is a holomorphic polynomial}
\right\}.
\end{equation}
The rational hull of $H$ is given by 
\begin{equation}
\hat{H}_{\text{rat}}=\left\{z\in\C^n:p(z)\in p(H) \text{ for all holomorphic polynomials}\right\}.
\end{equation}
$H$ is said to be polynomially convex if $\hat H=H$ and rationally convex if $\hat H_{\text{rat}}=H$.
\end{defn}

A different description of rational hull will be more useful in the context
of this paper. The rational hull of $H$ is the set of all points $z\in \C^n$ for
which one cannot find a complex algebraic hypersurface that passes through
$z$ and does not intersect $H$.

The fact that the polynomial hull contains the rational hull
follows immediately from the definitions.
In particular, if the rational hull contains an open neighborhood of $p$,
then so does the polynomial hull. Additionally,
if $H$ is not rationally convex, it is not polynomially convex either.

Polynomial convexity and rational convexity play an important role in generalizing Runge's approximation theorem to several variables.
The Oka-Weil theorem states that on compact, polynomially convex subsets of $\C^n$, holomorphic functions can be approximated uniformly by holomorphic polynomials. Likewise, holomorphic functions on a rationally convex set $H$ can be approximated uniformly by rational functions with poles off $H$. See, for example, Stout \cite{Stout}.

\begin{defn}
An analytic disc is (an image of) a non-constant holomorphic mapping
$\varphi \colon \D\to \C^n$.  If the map extends continuously to $\bar \D$ and
$\varphi(\partial \D)\subset H$, then we say that the disc $\varphi$ is attached
to $H$.
\end{defn}

As a consequence of the Maximum modulus theorem,
attached analytic discs are contained
in the polynomial hull.
In particular, if $\varphi(\D)$ is an attached disc to $H$,
then $\varphi(\D)\subset \hat{H}$.
If $H$ is not polynomially convex, a question about a structure of
$\hat{H}\setminus H$ or $\hat{H}$ arises. In particular, does
$\hat{H}\setminus H$ or $\hat{H}$ contain analytic structure, or can we
obtain $\hat{H}$ by attaching discs to $H$?  Stolzenberg~\cite{hullnodisc},
and many since, gave examples of polynomial hulls containing no analytic
discs. 

If we have a continuous family $\varphi_t$, $t \in [0,1]$,
of analytic discs attached to $H$
such that $\varphi_0$ is constant (and hence contained in $H$), then
$\varphi_1(\D) \subset \hat{H}_{\text{rat}}$.  This fact follows 
by, for example, applying the Argument principle to $f \circ \varphi_t$
for a polynomial $f$ that does not vanish on $H$.

\section{Support Curves in \texorpdfstring{$\C^2$}{C2}} \label{section:sc}

We begin by defining support curves, the central objects of this paper.

\begin{defn}
Let $H\subset \C^n$ be a hypersurface.
A holomorphic curve $C$ is called a support curve for $H$ at $p$ if there
exists a neighborhood $U$ of $p$ such that $C\cap U\cap H=\{p\}$. We say $H$
is curve-supported at $p$ if there exists a support curve for $H$ at $p$.
We say $H$ is curve-supported from two sides at $p$
if there exist two support curves $C_1$ and $C_2$ at $p$
and a neighborhood $U$ of $p$ such that $C_1 \setminus \{ p\}$
and $C_2 \setminus \{ p \}$ lie in distinct topological components
of $U \setminus H$.
\end{defn}

We will from now on say that a support curve $C$ is in a component of the
complement if $C \setminus \{ p \}$ is a subset of that component.
Let $V$ be a component of the complement $U \setminus H$ and $C \subset U$
be a holomorphic curve through $p$.
Then $C \setminus \{p\}$ is contained in $V$
if $C \setminus V \subset \{p\}$ or $C \subset V \cup \{p\}$.
A holomorphic curve that intersects two different components of the
complement $U \setminus H$ and intersects $H$ at at least one point
must intersect $H$ at infinitely many points.
Note that a complement of a singular hypersurface may have more than two components.

First, we focus on hypersurfaces in $\C^2$.
Let $H\subset \Delta \subset \C^2$ be a Levi-flat hypersurface
with Segre-degenerate point at the origin, and we fix a defining function $\rho$ as in
section~\ref{section:prelim}.
In $\C^2$, Segre varieties at non Segre-degenerate points are unions of complex curves. Using these curves, we show that $H$ is curve-supported at the origin.  

\begin{prop}\label{prop:onesc}
Suppose that $H\subset \C^2$ is a Levi-flat real-analytic hypersurface that is Segre-degenerate at $q$. Then $H$ is curve-supported at $q$.
\end{prop}

\begin{proof}
The result is local and hence we can work in a small neighborhood of $q$.
Suppose $q=0$, $H \subset \Delta \subset \C^2$, and $\pi_z$ and $\pi_w$
are as in section
\ref{section:prelim}.

Since $0$ is Segre-degenerate, every Segre variety will contain a component
passing through $0$.
We call the components containing $0$ the principal components.
We will show that there exists a principal component of a Segre
variety $\Sigma_p$ that does not intersect $H$ other than at $0$ and is thus
a support curve for $H$ at $0$. 
The set of points whose Segre varieties that intersect $H$ outside of $0$ is
$\pi_z\circ\pi_w^{-1}(H\setminus \{0\})$.  Let $Q=\pi_w^{-1}(H\setminus \{0\})$.

If a principal component of $\Sigma_p$ crosses $H$, then it must intersect
$H$ along a real curve. In particular, this component intersects $H$ on every
$K_\epsilon=H\cap \partial B_\epsilon(0)$ for $0\leq \epsilon<\delta$ for some
$\delta$. That is, by the symmetry property of Segre varieties,
\begin{equation}
p\in \bigcap_{0\leq \epsilon<\delta} \pi_z\circ\pi_w^{-1}(K_\epsilon)=\bigcap_{0\leq \epsilon<\delta}\pi_z(L_\epsilon),
\end{equation}
where $L_\epsilon=\pi_w^{-1}(K_\epsilon)$.

Now, $Q$ is a $5$-dimensional subset of $\mathcal H$. At each point of $Q^*$,
there is a local biholomorphism to $\pi_w^{-1}(M_2)$.
At other regular points $p$, there is a biholomorphism to $\pi_w^{-1}(M_1)$. 
The rank of $\pi_z$ is a real-analytic function that is constant in a neighborhood of a regular point. The rank of $\pi_z$ for $\pi^{-1}_w(M_j)$ is $2j-1$, and therefore the rank of $\pi_z$ at each regular point of $Q$ is at most $3$. 

Since $L_\epsilon\subset Q$, the rank of $\pi_z$ at regular points of
$L_\epsilon$ is at most $3$ as well. Cover $(L_\epsilon)^*$ with countably
many open sets containing regular points. On each such open set $\Omega$,
$\dim_\R \pi_z(\Omega)\leq 3$.  Furthermore, the dimension of the singular set
$(L_\epsilon)^{\text{sing}}$ is at most $3$. Therefore,
 $\dim_\R \pi_z\bigl( (L_\epsilon)^{\text{sing}}\bigr)\leq 3$.
Combining these two observations, we see
that $\pi_z(L_\epsilon)$ is a meagre set for every $\epsilon>0$.

Let $N=\{\frac1n:n\in \N\}$. Any point whose principal component intersects $H$ outside of $0$ will be contained in $\pi_z(L_t)$ for some $t\in N$ and therefore also in
\begin{equation}
S=\bigcup_{t\in N} \pi_z(L_t).
\end{equation}
The set $S$ is meagre as it is a countable union of meagre sets. 
Therefore, most principal components are in fact a support curves for $H$ at
$0$ in the sense that $\bigl(\bigcup_{t\in N} \pi_z(L_t)\bigr)^c$ being co-meagre.
\end{proof}


The existence of one support curve in one component of the
complement imposes a condition on the curves contained in other components
of the complement, if such curves exist.

\begin{lemma} \label{lemma:twosidedsupport}
Suppose that $H \subset U \subset \C^2$ is a Levi-flat hypersurface
curve-supported at $q$ and the support curve is contained in a component
$\Omega$ of the complement $U \setminus H$.
Then there exists a neighborhood $B_\epsilon(q)$ of $q$ such that every
complex curve $D \subset \Omega^c$ that is closed in $U$
and intersects $\Omega^c\cap B_{\epsilon}(q)$ intersects $H$ at $q$.
\end{lemma}

In particular, if $D \setminus \{ q \}$
is contained in a component of the complement $U \setminus H$, then
$q \in D$ and hence, $D$ is a support curve for $H$.  Note that the lemma
still holds even if $D$ is a subset of $H$.

\begin{proof}
Again, we can work in a small neighborhood of $q$.
Suppose $q=0$, $H \subset \Delta \subset \C^2$.

As $C$ is a complex curve, $V=B_\delta(0)\setminus C$
(for small enough $\delta > 0$) is pseudoconvex.
Therefore, there exists a plurisubharmonic function $P$ on $V$ that blows up on the boundary $\partial V=C\cup \partial B_\epsilon(0)$.
In fact, we may take $P=-\log \delta(z)$ where $\delta(z)$
is the distance to the boundary function.

Let $D$ be a complex curve, closed in $\Delta$,
that does not intersect the component of the complement $\Delta \setminus H$
containing $C$ and suppose $0 \notin D$.

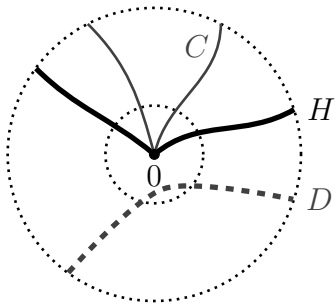
\begin{center}
\begin{tikzpicture}[scale=1.3]

\draw[line width=1pt,dotted] (0,0) circle [radius=0.5];
\draw[line width=1pt,dotted] (0,0) circle [radius=1.5];

\draw[line width=2pt] (-1.21,0.88) to [out=-45,in=140] (0,0);
\draw[line width=2pt] (0,0) to [out=40,in=210] (1.43,0.46) node[right] {$H$};

\draw[darkgray,line width=1pt] (-0.681,1.336) to [out=-45,in=105] (0,0);
\draw[darkgray,line width=1pt] (0,0) to [out=77,in=270] (0.68,1.34) node[below left] {$C$};

\draw[darkgray,line width=2pt,dashed] plot [smooth] coordinates {(-0.88,-1.21) (0.1,-0.35) (1.42,-0.46)} node[right] {$D$};

\filldraw (0,0) circle (0.05) node[below] {$0$};

\end{tikzpicture}
\end{center}

Let $B_\epsilon(0)$ be a smaller neighborhood of $0$ containing a subset of $D$.
Restricting $P$ to $D$, gives a plurisubharmonic function with the maximum
occurring at the closest point to $0$, which is an interior point.
The restriction $P|_D$ cannot be constant as $D$ approaches
$\partial B_\delta(0)$ outside of $B_\epsilon(0)$ so $P|_D$ blows up.
This is a contradiction of the maximum modulus principle. 
\end{proof}

For Levi-flat hypersurfaces the Levi foliation guarantees the existence of
curves contained in every component of the complement of $H$ as well as
curves contained in $H$. Near a regular point $p$ of $H$ with $\dim_p H=3$, a
biholomorphism with $M_2$ gives a local foliation of a neighborhood of $p$.
A leaf is contained in a component of Segre variety and does not intersect
$H$ in that neighborhood. As $p \to 0$ the local foliation gives
curves whose closure intersects $H$ at most at $0$. 

Let $L$ be a leaf of a Levi foliation of $H$.
Then $L$ is contained in a component $D$ of $\Sigma_p$ contained in $H$.
Applying Lemma \ref{lemma:twosidedsupport} to $D$ shows that
$0\in D$ and therefore $0\in \bar L$.
In other words, $0$ is a dicritical singularity.

Once a support curve is constructed in the second component of the complement,
Lemma~\ref{lemma:twosidedsupport} applies to curves in
the first component of the complement as well.

We can summarize the two Lemmas and the discussion above with the following theorem:

\begin{thm} \label{thm:TFAE degen}
Let $H\subset \C^2$ be a Levi-flat real-analytic hypersurface. The following are equivalent
    \begin{enumerate}
        \item $p\in H$ is a Segre-degenerate point,
        \item $H$ is curve-supported at $p$, and
        \item $H$ is curve-supported from two sides at $p$.
        \item $p$ is a dicritical singularity of $H$.
    \end{enumerate}
\end{thm}

\begin{remark}
Strongly pseudoconvex hypersurfaces are curve-supported at $p\in H$,
but Lemma \ref{lemma:twosidedsupport} does not apply, as there are
no curves contained in the strongly pseudoconvex component of the complement.
\end{remark}

\section{Support Curves in \texorpdfstring{$\C^n$}{Cn}}\label{section: Cn}


A tempting strategy to extend the results to $\C^n$ is to try to restrict
to a 2-dimensional subspace, as the intersection of a Levi-flat hypersurface
with a complex subspace is still Levi-flat (if it is a hypersurface),
and we have already proved a result in $\C^2$.
One needs to be careful, however, when applying Segre techniques 
via restricting subvarieties in $\C^n$ to lower dimensional complex manifolds.
In particular, the restriction of a Segre-degenerate hypersurface
may not be Segre-degenerate.
The reason is that the restriction of the defining
function of $H$ need not generate the ideal of functions vanishing on the
restricted hypersurface.
If that is the case, restrictions of Segre
varieties with respect to $\rho$ need not be Segre varieties with respect to
the defining function of the restriction.

The simplest example of this phenomenon is to restrict 
$H$ to a $1$-dimensional complex plane $P$.
The set $H \cap P$ is in general a real curve in a $1$-dimensional 
complex plane, and such a set is never Segre-degenerate.
If $f(z,\bar z)$ vanishes on $P\cap H\subset \C^1$
such that $f(z,0)=f(0,\bar z)=0$, then
both $z$ and $\bar z$ divide $f$.
Therefore, $\lvert z\rvert^2$ divides $f$ as
well, and $f$ is not a defining function for $P\cap H$. 

However, we can still use the results and techniques used in $\C^2$ to our
advantage. Let $H$ be a Levi-flat hypersurface in $\C^n$ that is curve
supported at $p$. By pulling back via a finite map and restricting
a different hypersurface, we construct a Levi-flat hypersurface
$\widetilde H$ in $\C^2$
which is curve supported at $\widetilde p$ and thus $\widetilde p$ is a
Segre-degenerate point of $\widetilde H$. The construction of  $\widetilde H$ then
guarantees $p$ is a Segre-degenerate point of $H$. 

\begin{prop}\label{Prop:restriction of support curve} 
Let $H\subset U \subset \C^n$ be a Levi-flat hypersurface
that is curve-supported at $q$.
Then $q$ is Segre-degenerate point,
and $H$ is curve-supported from two sides at $q$.
\end{prop}

\begin{proof}
As before, we can work in small neighborhood of $q$ and may assume $q=0$.

Let $C$ be a support curve of $H$ at $0$.
By the Puiseux theorem, after possibly a rotation, $C$ can be
parameterized by
$t \mapsto \bigl(t^k, f_2(t),\dots,f_n(t)\bigr)$ for some integer $k$ and
holomorphic functions $f_2,\dots,f_n$ satisfying $f_i(0)=0$. 

Consider a finite map $\varphi \colon \C^n\rightarrow \C^n$ defined by 
\begin{equation}
\varphi(z_1,\dots,z_n)=\bigl(z_1^k,z_2+f_2(z_1),\dots,z_n+f_n(z_1) \bigr),
\end{equation}
and let $G=\varphi^{-1}(H)$. Then $G$ is a Levi-flat hypersurface.
Let $L=\{z_2=\dots =z_n=0\}$ be a complex line. Then $\varphi(L)=C$ and $L\cap G=\{0\}$ as
\begin{equation}
\varphi(L\cap G)\subset \varphi(L)\cap \varphi(G)=H\cap \varphi(\varphi^{-1}(H))\subset C\cap H\subset\{0\}
\end{equation}
and clearly $0\in L\cap G$. That is, $L$ is a support curve of $G$ at $0$.

Restricting $G$ to a $2$-dimensional complex plane
$P=\{z_3= \dots =z_n=0\}$,
containing $L$, we can view $G\cap P$ as a Levi-flat hypersurface in $\C^2$,
which is curve-supported at $0$.
By Theorem~\ref{thm:TFAE degen}, $G\cap P$ is Segre-degenerate at $0$ and
curve-supported from two sides at $0$.  Embedding the support curves back to
$\C^n$ shows that $G$ is curve-supported from two sides at~$0$. To see that
$G$ is Segre-degenerate, notice that Segre varieties of $G\cap P$ lift to
distinct Segre varieties of $G$ and therefore infinitely many distinct
components of  Segre varieties contain $0$, i.e., $0$ is a Segre-degenerate
point of $G$.

Finally, we show that $H$ has the desired properties as well.
The image of two support curves at $0$ in two different components of $G$
will be support curves at $0$ in two different components of $H$.
Let $r$ be the defining function of $G$ at $0$ and $\rho$ the defining
function of $H$.
Since $0$ is a Segre degenerate point of $G$, $r(z,0)\equiv 0$.
Then $\rho \circ \varphi = ar$ for some real-analytic $a$, and
\begin{equation}
0\equiv a(z,0) r(z,0)=\rho(\varphi(z),\overline{\varphi(0)})=\rho(\varphi(z),0).
\end{equation}
As $\varphi(z)$ is finite, $\rho(z,0)\equiv 0$ and $0$ is a
Segre-degenerate point of $H$.
\end{proof}

Next, suppose $p$ is a Segre-degenerate point of $H$.
Restricting principal components of Segre varieties to
a $2$-dimensional manifold $V$ containing the origin
produces holomorphic curves through the origin.
We show that most of them are, in fact, support curves.

\begin{prop}
Let $H\subset \Delta \subset \C^n$ be a Levi-flat hypersurface,
which is Segre-degenerate at $q$. Then $H$ is curve-supported at $q$.
\end{prop}

\begin{proof}
Again, consider a $H \subset \Delta \subset \C^n$, $\rho$, and $q=0$ as usual.
Segre-degenerate points of $H$ (with respect to $\rho$)
are contained in some complex subvariety $S$
of dimension at most $n-2$. Let $X$ be a $2$-dimensional complex manifold
that intersects $S$ transversally at $q$.
That is, $q$ is the only
Segre-degenerate (in any sense) point of $H$ in $X$. 

Adapting the proof of Proposition~\ref{prop:onesc} to $\widetilde H=H\cap X$
and $\widetilde K_\epsilon=K_\epsilon\cap H\cap X$ shows that there exists a
point $p$ and a principal component $P_p\subset \Sigma_p$ such that $P_p\cap
X$ is a support curve of $H$ at $0$.
\end{proof}

Putting the results together, we obtain a characterization of
Segre-degenerate points of a Levi-flat hypersurface.
Here we mean Segre-degenerate in the germ sense, of course, as we have not
fixed any defining function.

\begin{thm}\label{thm:TFAE Cn}   Let $H\subset \C^n$ be a Levi-flat real-analytic hypersurface, where $n\geq2$. The following are equivalent
    \begin{enumerate}
        \item $p\in H$ is a Segre-degenerate point,
        \item $H$ is curve-supported at $p$, 
        \item $H$ is curve-supported from two sides at $p$, and
        \item $p$ is a dicritical singularity of $H$.
    \end{enumerate}
\end{thm}

Given the introductory paragraph to this section, a corollary worth
mentioning is the following:

\begin{cor}
Let $H\subset U \subset \C^n$ be a Levi-flat hypersurface
with $S$ the set of Segre-degenerate (germ sense) points of $H$.
Let $X$ be a $k$-dimensional, $k\geq 2$, complex submanifold of $\C^n$ containing
$p\in S$ such that $X \cap H$ is $(2k-1)$-dimensional and $X \cap S$ is
contained in a $(k-2)$-dimensional complex subvariety. Then $X\cap H$ is Segre-degenerate (germ sense).
\end{cor}

Note that most complex submanifolds $X$ through $p$ of dimension at least 2
satisfy the hypotheses.
    
\section{Applications of Support Curves}\label{section:polyhulls}

As an application, we show that hypersurfaces that are curve-supported from two sides have interesting properties. Let $H\subset U\subset \C^n$ be a hypersurface, not necessarily Levi-flat, that is curve-supported from two sides at $p$.
Support curves lend themselves naturally to the construction of polynomial and rational hulls. 

The existence of support curves allows us to say even more about the analytic structure of these hulls. 

\begin{thm}\label{thm:hullopennbhd}
Let $H\subset U\subset \C^n$ be a hypersurface which is curve-supported at
$p \in H$ and the support curve is contained in the component $V$ of the
complement $U \setminus H$.
Then for some small enough $\epsilon > 0$, there exists an analytic disc in
$U$ attached to $H$ through every $q\in V^c\cap B_\epsilon(p)$.
\end{thm}

\begin{proof}
Let $C$ be a support curve in the component $V$ of the complement.

There exists a small enough neighborhood $\Omega$ of $p$ such that for all $v\in
B_\epsilon(0)$ such that $p+v\in V^c$. Translating $C$ by $v$ yields a curve
$C+v$ through $p+v$ that intersects two components of the complement of $H$.
In particular, $(C+v)\cap H$ is a real curve.  Moreover, by continuity, for
small $v$ this curve is compact.

Parameterize $C$ with $\varphi \colon \D \to \C^n$ and
let $A=\{z\in\D: \varphi(z)\in H\}$. 

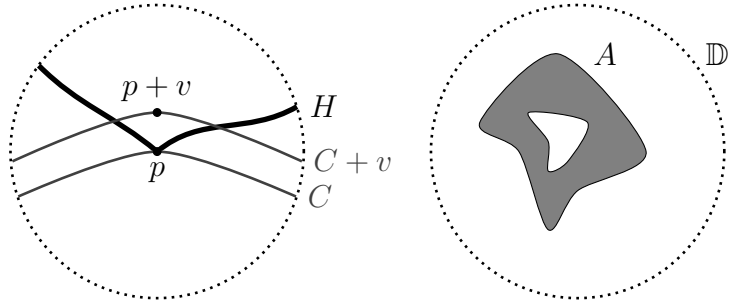
\begin{center}
\begin{tikzpicture}[scale=1.3]

\draw[line width=1pt,dotted] (0,0) circle [radius=1.5];

\draw[line width=2pt] (-1.21,0.88) to [out=-45,in=140] (0,0);
\draw[line width=2pt] (0,0) to [out=40,in=210] (1.43,0.46) node[right] {$H$};

\draw[darkgray,line width=1pt] plot [smooth] coordinates {(-1.42,-0.46) (0,0) (1.42,-0.46)} node[right] {$C$};

\draw[darkgray,line width=1pt] plot [smooth] coordinates {(-1.48,-0.10) (0,0.4) (1.48,-0.10)} node[right] {$C+v$};

\filldraw (0,0.4) circle (0.04) node[above] {$p+v$};
\filldraw (0,0) circle (0.04) node[below] {$p$};

\draw[line width=1pt,dotted] (4.3,0) circle [radius=1.5];
\draw (5.5,1) node [right] {$\mathbb{D}$};

\draw[fill=gray] plot [smooth cycle] coordinates {(4,-0.8) (4.3,-0.4) (5,0) (4.1,1) (3.3,0.3) (3.7,0)}; 
\draw[fill=white] plot [smooth cycle] coordinates {(4,-0.2) (4.2,-0.1) (4.4,0.3) (3.8,0.4) (4,0.1) }; 

\draw (4.35,1) node [right] {$A$};

\end{tikzpicture}
\end{center}

Filling any holes if necessary, $A$ bounds a simply connected open set and
is, therefore, biholomorphic to $\D$ by the Riemann mapping theorem.
For small enough $v$, the image of the interior of $A$ stays within $U$.
In other words, restriction of $C+v$ gives a disc attached to $H$ through
$p+v$ and contained in $U$.
\end{proof}

Even more can be said about hypersurfaces that are curve-supported from two sides:
\begin{cor}\label{cor:polyhullsattacheddisc}
Let $H \subset U \subset \C^n$ be a hypersurface that is curve-supported from two-sides
at $p \in H$. Then, there exists a neighborhood $B_\epsilon(p)$ of $p$ such
that every point $q\in B_\epsilon(p)$ is contained in an analytic disc
in $U$ attached to $H$.
\end{cor}

Examining the proof in greater detail, as $v\rightarrow p$ the discs form a continuous family of attached discs shrinking to $p$. An application of the Argument principle then gives the following corollary:
\begin{cor} 
Let $H\subset U\subset \C^n$ be a hypersurface and $p \in H$. Then:
\begin{enumerate}
    \item If $H$ is curve-supported at $p$ and the support curve lies in
          a component $V$ of the complement $U \setminus H$,
          then $V^c\cap B_\epsilon(p)\subset
           \hat H_{\text{rat}}\subset \hat H$ for $\epsilon > 0$ small enough.

    \item If $H$ is curve supported from two sides at $p$, then
          $B_\epsilon(p) \subset \hat H_{\text{rat}}\subset \hat H$
          for some $\epsilon > 0$ small enough.
\end{enumerate}
\end{cor}

The existence of analytic discs shrinking to a point is crucial for
extending CR functions, by the use of the Kontinuit\"atssatz, which comes in
many forms.  The form we use is the following
(see, e.g.,\ \cite{Ivashkovich:cont}).
By a sequence of sets converging, we mean in the
Hausdorff sense, that is, eventually the elements of the sequence
are within an $\epsilon$ neighborhood of the limit set and vice versa.

\begin{lemma}[Kontinuit\"{a}tssatz or Behnke--Sommer]
Suppose
$\Omega\subset \C^n$ is open and $\{D_k\}_{k=1}^\infty$ is a sequence
of analytic discs in $\Omega$ converging to $D \subset \C^n$
such that $\partial D_k$ converges to $\Gamma \subset\subset \Omega$.
Then $D$ lifts to the envelope of holomorphy, that is,
any $f \in \sO(\Omega)$ can be analytically continued to a neighborhood of
$D$.
\end{lemma}

Using the families of attached discs previously constructed
and applying Kontinuit\"{a}tssatz,
we prove the following extension property of holomorphic functions.    

\begin{cor} 
Let $H$ be a hypersurface that is curve-supported from two sides at $p$.
Then there exists a neighborhood $\Omega$ of $p$ such that every
$f\in \sO(H)$ can be extended to a holomorphic function on $\Omega$,
that is, there exists an $F \in \sO(\Omega)$ such that
$F|_{\Omega \cap H} = f|_{\Omega \cap H}$.
\end{cor}

\begin{proof}
Let $\Omega = B_\epsilon(p)$ be the neighborhood of $p$ from Corollary~\ref{cor:polyhullsattacheddisc}.
That is, for all $q\in \Omega$, there exists a disc $\varphi_q(\D)$ through $q$ that is attached to $H$.
Moreover, such discs come from a continuous family shrinking to $p$.

Let $f \in \sO(H)$, that is, $f$ is holomorphic on some neighborhood $U_f$
of $H$.  The small discs from the family that shrinks to $p$ are
eventually inside $U_f$.
Hence, by the Kontinuit\"{a}tssatz, $f$ extends holomorphically to a neighborhood of
any point in $\varphi_q(\D)$, in
particular, to a neighborhood of $q$ via analytic continuation 
via a path from $p$.
By construction, as the disks are created by translating the support curves,
we can assume that these paths are going away from $p$
in the sense that they are transverse to any sphere centered at $p$.
Such paths fill the ball $\Omega = B_\epsilon(p)$.
We claim that $f$ has a unique holomorphic extension to $\Omega$.
It clearly extends to some ball $B_{\delta}(p) \subset \Omega$.
If $\delta < \epsilon$, then we can cover the sphere $\partial B_\delta(p)$
with small balls through which the function continues uniquely
since it continues along a paths that pass through the sphere
transversally outward, and we hence get unique continuation
to a larger ball, and hence to all of $\Omega$.

\end{proof}

Other versions of such extensions can be proved.
For example, 
the components of $U \setminus H$ for a curve supported $H$ have the same extension property as
a Hartogs triangle.  There is some neighborhood $\Omega$ of $p$
such that if $f$ is holomorphic in a component $V$ of the complement $U \setminus H$
in which $H$ is curve supported,
and $f$ extends holomorphically to an arbitrarily small neighborhood of $p$, then $f$ extends 
holomorphically to all of $\Omega$.  To prove this statement with the
theorem, one needs to perturb $H$ to make it fit within $V$.  This
perturbation is not
difficult as this perturbation need not be real-analytic in order to apply
the proof of the theorem, it only needs to be a hypersurface that preserves
the property of being curve supported by the same curve $C$.


\def\MR#1{\relax\ifhmode\unskip\spacefactor3000 \space\fi%
  \href{http://mathscinet.ams.org/mathscinet-getitem?mr=#1}{MR#1}}


\end{document}